\documentclass[a4paper,12pt]{article}
\usepackage{bookman}
\usepackage{xcolor}
\usepackage[T1]{fontenc}
\usepackage{hyperref}
\usepackage{tgbonum}
\usepackage{setspace} 
\usepackage{tikz}
\usepackage[utf8]{inputenc}
\usepackage{amsfonts}
\usepackage{amssymb}
\usepackage{amsmath}
\usepackage{latexsym}
\usepackage{enumerate}
\usepackage{amsthm}
\newtheorem{lem}{Lemma}[section]
\newtheorem{rmk*}{Remark}[section]
\newtheorem{thm}{Theorem}[section]

\begin{document}
\title{\large Almost resolvable even cycle systems of $(K_u \times K_g)(\lambda)$}
\author{{\small S. Duraimurugan, A. Shanmuga Vadivu and  A. Muthusamy  } \\ {\footnotesize Department
of Mathematics, Periyar University, Salem,} \\ {\footnotesize
Tamil Nadu, India.} \\ {\footnotesize \emph{ E-mail:  duraipu93@gmail.com, avshanmugaa@gmail.com and appumuthusamy@gmail.com }}}
\date{}
\maketitle
\begin{center}
\bf Abstract
\end{center}
In this paper, we  prove that almost resolvable $k$-cycle systems (briefly $k$-ARCS) of $(K_u \times K_g)(\lambda)$ exists for all  $k \equiv 0(mod \ 4) $ with  few possible exceptions,  where $\times$ represents tensor product of graphs.

\section{Introduction} 
\par Throughout this paper all the graphs considered here are  finite. For a graph $G$, we use $G(\lambda)$ and $\lambda G$, where $\lambda$ is an integer, to represent the multi-graphs obtained from $G$ by replacing each edge of $G$ with uniform edge-multiplicity $\lambda$ and $\lambda$ edge-disjoint copies of $G$ respectively. Let $C_t$, $K_t$ and $\bar{K}_{t}$ respectively, denote the {\it cycle, complete graph} and {\it complement of complete graph} on $t$ vertices. Let $K_{t,t}$ be a {\it complete bipartite graph} with bipartition $(A,B)$, where $A=\{a_0,a_1,...,a_{t-1}\} \ \mbox{and} \ B=\{b_0,b_1,...,b_{t-1}\}$. The edge set of $F_i(A,B)\subset K_{t,t}$ is defined as $\{a_jb_{j+i}: \ 0 \leq j \leq t-1 \}$, $0 \leq i \leq t-1$, where addition in the subscripts are taken modulo $t$. Clearly $F_i(A,B)$ is a {\it $1$-factor of $K_{t,t}$ with distance $i$ }from $A$ to $B$. Also, $\oplus_{i=0}^{t-1}F_i(A,B)=K_{t,t}.$ 
\par For two graphs $G$ and $H$, their {\it lexicographic product} $G \otimes H$ has the vertex set $V(G \otimes H)=V(G) \times V(H)$ in which two vertices  $(g_1,h_1)$ and $(g_2,h_2)$ are adjacent whenever $g_1g_2 \in E(G)$ or $g_1=g_2$ and $h_1h_2 \in E(H)$. Similarly, the {\it tensor product} $G \times H$ of two graphs $G$ and $H$ has the vertex set $V(G) \times V(H)$ in which two vertices $(g_1,h_1)$ and $(g_2,h_2)$ are adjacent whenever $g_1g_2 \in E(G)$ and $h_1h_2 \in E(H)$. Clearly, the tensor product is commutative and distributive over edge disjoint union of graphs, i.e, if $G=G_1 \oplus G_2 \oplus...\oplus G_{u},$ then $G \times H=(G_1 \times H) \oplus...\oplus (G_u \times H)$. One can easily observe that $K_u \otimes \bar{K}_{g} \cong K_{g,g,....g}$, the complete $u$-partite graph in which each partite set has $g$ vertices. Hereafter, we denote a {\it complete $u$-partite graph} with $g$ vertices in each  partite set as $K_u \otimes \bar{K}_{g}$.
It is clear that $(K_u \otimes \bar{K}_{g})-gK_u \cong K_u \times K_g$, where $gK_u$ denotes $g$ disjoint copies of $K_u$.   For more details  about product graphs, the readers are referred to \cite{WI}.
\par A {\it $k$-factor} (respectively, {\it near $k$-factor}) of $G$ is a  $k$-regular spanning subgraph of $G$ (respectively, $G \setminus \{v\}$, for some  $v \in V( G)$). Note that the {\it $C_k$-factor} (respectively, {\it near $C_k$-factor}) of $G$ can also be called a $2$-factor of $G$ (respectively, $G \setminus \{v\}$, for some $v \in V(G)$), where the components are $C_k$.  We say that the graph $G$ has a {\it decomposition} into $G_1, G_2,...,G_r$, for some integer $r \geq 1$, if $G= \oplus_{i=1}^r G_i$ and $G_1, G_2,...,G_r$ are pairwise edge-disjoint subgraphs of $G$. We say that $G$ has a { \it $C_k$-decomposition} if $G$ can be partitioned into edge-disjoint cycles of length $k$, and the existence of such decomposition is denoted as $C_k \rvert G$.
A {\it $1$-factorization} (respectively, near $1$-factorization) of $G$ is the partition of $G$ into edge-disjoint $1$-factors (respectively, near $1$-factors).
 A {\it $C_k$-factorization} of $G$ is the partition of  $G$ into edge-disjoint $C_k$-factors, denoted by $C_k \rVert G$. A near {\it $C_k$-factorization} of $G$ is the partition of  $G$ into edge-disjoint {\it near $C_k$-factors}.
A {\it partial $k$-factor}  of $(K_u \otimes \bar{K}_{g})(\lambda)$ is the $k$-factor of $(K_u \otimes \bar{K}_{g})(\lambda) \setminus V_i$, for some $i \in \{1,2,3,..,u\}$, where $V_1, V_2,V_3,...,V_{u}$ are the partite sets of $(K_u \otimes \bar{K}_{g})(\lambda)$.
A {\it partial $k$-factorization} (respectively, {\it partial $C_k$-factorization} )
$(K_u \otimes \bar{K}_{g})(\lambda)$ is a decomposition of $(K_u \otimes \bar{K}_{g})(\lambda)$ into partial $k$-factors (respectively, partial $C_k$-factors).
 Let $K$ be a set of integers.
A resolvable $K$-cycle systems (respectively, almost resolvable $K$-cycle systems), briefly $K$-RCS (respectively, $K$-ARCS) of $(K_u \otimes \bar{K}_g)(\lambda)$ is a decomposition of $(K_u \otimes \bar{K}_g)(\lambda)$ into $2$-factors (respectively, partial $2$-factors) consists cycles of lengths from $K$. When $K=\{k\},$ we write $K$- RCS as a $k$-RCS and $K$-ARCS as a $k$-ARCS.
A $(k,\lambda)$-modified cycle frame, briefly $(k,\lambda)$-MCF of $(K_u \otimes \bar{K}_{g})(\lambda)$ is a  decomposition of $(K_u \otimes \bar{K}_{g})(\lambda)-gK_u(\lambda)$   into partial $C_k$-factors. 
It is appropriate to mention that $k$-ARCS of $(K_u \times K_g)(\lambda)$  is nothing but a modified cycle frame. One can easily observe that $(K_u \times K_g)(\lambda)$ is an $u$-partite graph with edge multiplicity $\lambda$ and $g$ vertices in each part ; it can also be considered as an $g$-partite graph with edge multiplicity $\lambda$ and $u$ vertices in each part by the commutative property of tensor product. Hereafter, we consider $(K_u \times K_g)(\lambda)$ as a $u$-partite graph with edge multiplicity $\lambda$ and $g$ vertices in each part. 
\par RCS/ARCS research has a direct link to various types of cycle frames. Many mathematicians (Stinson \cite{DR}, Cao et al. \cite{HC2,HC1}, Niu et al. \cite{co}, Chitra et al. \cite{AM}, Muthusamy et al. \cite{AM2} Buratti et al. \cite{MB}) have studied cycle frames, which have connections to many well-known combinatorial problems such as the Oberwolfach problem and the Hamilton waterloo problem. The above facts motivate us to work on RCS/ARCS in this paper.
\par D.R. Stinson \cite{DR} proved that there exists a $3$-ARCS of  $(K_u \otimes \bar{K}_{g})(\lambda)$ if and only if $u \geq 4, \ \lambda g \equiv 0 \ (mod \ 2)$ and $g(u-1)\equiv 0 \ (mod \ 3)$. M. Niu and H. Cao
\cite{co} proved that there exits a $\{3,5\}$-ARCS of $(K_u \otimes \bar{K}_{g})(\lambda)$. Lot of results have been published on  ARCS of $(K_u \otimes \bar{K}_g)(\lambda),$ see  \cite{MB,HC1,AM,AM2}.  RCS of $(K_u \times K_g)(\lambda)$ has been studied by many authors \cite{AB,AS,AS1,GE}.
$\mbox{H. Cao et al.}$ \cite{HC2} proved that there exists a $3$-ARCS of $(K_u \times {K}_{g})(\lambda)$ if and only if $\lambda (g-1) \equiv \ 0 (mod \ 2 )$, $g(u-1) \equiv 0 (mod \ 3)$, $g \geq 3$, and $u \geq 4$, except for $(\lambda,g,u) =\{(1,3,6),(2m+1,6n+3,6)$, $m \geq 0$ and $n \geq 1\}.$ \\
In this paper, we shall prove that the necessary conditions for a $k$-ARCS of $(K_u \times {K}_{g})(\lambda)$ are also sufficient with few possible exceptions.
\par The necessary conditions for a  $k$-ARCS of $(K_u \times {K}_{g})(\lambda)$ are shown in the following Theorem.
\begin{thm} \label{nsty}
For all even integers $ k \geq 4$, if $(K_u \times K_g)(\lambda)$ has a $k$-ARCS, then
\vspace{-.2cm}
\begin{itemize}
\itemsep -.5em 
\item[(i)] $  u \geq 3$ and  $ g \geq 2$ \item[(ii)] $g(u-1) \equiv 0 (mod \ k)$ \item[(iii)]  $ \lambda(g-1) \equiv 0(mod \ 2)$.
\end{itemize}
\end{thm}
\vspace{-.4cm}
\begin{proof}
Since $ k \geq 4$ is an even integer, it is clear from the definition of  $k$-ARCS that $ u \geq 3$ and $ g \geq 2$. As the existence of $k$-ARCS gives the edge disjoint union of partial $C_k$-factors of $(K_u \times K_g)(\lambda),$ the number of vertices  in $(K_u \times K_g)(\lambda) \setminus V_i,$ for some $i \in \{1,2,...,u\}$, where $V_1,V_2,...,V_u$ are the partite sets of $(K_u \times K_g)(\lambda) $ must be divisible by $k,$ so $g(u-1) \equiv 0(mod \ k).$  Since each partial $C_k$-factor of $(K_u \times K_g)(\lambda)$ consists of $g(u-1)$ edges,
the number of partial $C_k$-factors in $(K_u \times K_g)(\lambda)$ is 
$$ \lambda \frac{\frac{u(u-1)}{2}g^2- \frac{u(u-1)}{2} g}{(u-1)g}=\lambda \frac{u(g-1)}{2}.$$
Hence, there are precisely $ \frac{\lambda(g-1)}{2}$ partial $C_k$-factors correspond to each partite set $V_i, \ i \in \{1,2,...,u\} $.
\end{proof}
\begin{thm} \label{mt}
	For even $k \equiv 0(mod \ 4),$	there exists a $k$-ARCS of $(K_u \times {K}_{g})(\lambda)$ if and only if $u \geq 3$, $g\geq 2$, $\lambda(g-1) \equiv 0 (mod \ 2)$, $g(u-1) \equiv 0 (mod \ k),$   except possibly $(\lambda,k,u) \in \{(2s,4,4x), (2s,4t,8), (2s,4t,4x+2),\ s,t,x \geq 1\}$ and $(\lambda,k,u,g) \in \{(2x+1,4t,8t+1,y), (2x+1,rs,2r+1,sy), \ x \geq 0, \ t \geq 1,  \ r \ \mbox{is even}, \ s\geq 3 \ \mbox{and} \ y \geq 1\ \mbox{are odd} \}.$
\end{thm} 
\section{Preliminaries}
To prove our results we need the following:

\begin{thm}\cite{ob} \label{op} 
	For any odd integer $t \geq 3$, if $u \equiv t(mod \ 2t)$, then $C_t||K_u$.
\end{thm}
\begin{thm}\label{t1} \cite{MI}
	For $u \geq 3$, $K_u$ has a near $1$-factorization if and only if $ u\equiv 1 (mod \ 2)$.	
\end{thm}
\begin{thm} \cite{JB}\label{t4} 
	For $u \geq 3$, $K_u(2)$ has a near $C_{2k}$-factorization if and only if $u \equiv 1 (mod \ 2k)$.
\end{thm}
\begin{thm} \cite{MM}\label{t6}
	There exists a  $C_{2m}$-factorization of $K_u(\lambda)$ if and only if $\lambda$ is even and $u \equiv 0 (mod \ 2m)$.	
\end{thm}
\begin{thm} \cite{KH}\label{t5}
	For any odd $m \geq 3$ and  a non-negative integer $s$, there exists a near $C_m$-factorization of $K_{ms+1}(2)$.
\end{thm}
\begin{thm}\label{mn} \cite{HI}
	$K_{m,n}$ has a $C_k$-factorization if and only if
	\vspace{-.2cm}
	\begin{itemize}
		\itemsep -.5em  
		\item [(i)] $m=n \equiv 0 (mod \ 2)$
		\item [(ii)] $k \equiv 0 (mod \ 2) \geq 4$
		\item [(iii)] $2n \equiv 0 (mod \ k)$
	\end{itemize}
	with precisely one exception, namely $m=n=k=6$.
\end{thm}
\begin{thm} \label{hc}\cite{PP2}
	For $m,n \geq 1$, the graph $C_{2n+1} \times K_{2m}$ has a Hamilton cycle decomposition.
\end{thm}
\begin{thm} \label{pten}\cite{MP}
	The graph $C_{2n} \times K_{2m}$ has a Hamilton cycle decomposition.
\end{thm}
\begin{thm} \label{hc2} \cite{la}
	The graph $C_m \otimes \bar{K}_{n}$ has a Hamilton cycle decomposition.
\end{thm}
\begin{thm} \cite{PP2}\label{l1}
	If $C_k \rVert G$ and $n|m$, then $C_{kn} \rVert G \times K_m$, where  $m \not \equiv  2(mod \ 4),$ when $k$ is odd.
\end{thm}
\begin{thm}\cite{PP} \label{f}
	For an odd integer $m \geq 3$ and even integer $k \geq 4,$ we have $C_k \rVert G^k \times K_m$ except possibly for $(k,m)=(4,3),$ where $G^k$ is a cubic graph of order $k.$
\end{thm}
\begin{thm}\cite{HC1}\label{6}
	There exists a $6$-ARCS  of $K_4 \otimes \bar{K}_2$.
\end{thm}
\begin{thm}\cite{HS}\label{ttt}
	$K_{t,t,t}$ has a $C_t$-factorization.
\end{thm}
\begin{thm}\cite{MI}\label{k2}
	There exists a partial $1$-factorization of $K_u \otimes \bar{K}_g,$ if and only if $u \geq 3$ and $g(u-1) \equiv 0 \ (mod \ 2)$.
\end{thm}
\section{Basic Constructions}
\begin{thm}\label{I1}
	There exists a partial $C_k$-factorization of $K_{k+1} \times K_t$, for all odd $t \geq 3$  and $k \equiv 0(mod \ 4)$.
\end{thm}
\begin{proof}
	To prove this Theorem, without loss of generality we consider a  graph $K_{k+1}(2)$ as follows. Let $V(K_{k+1}(2))=\{a_0,a_1,a_2,...,a_{k-1},a_{\infty}\}.$ By Theorem  \ref{t4}, $K_{k+1}(2)$ has a near $C_k$-factorization. The required near $C_k$-factors are given by\\
	$G _i=(a_i,a_{i+1},a_{i-1},a_{i+2},a_{i-2},...,a_{i+\frac{k}{2}-1},a_{i-\frac{k}{2}+1},a_{\infty}),$ $0 \leq i \leq k-1$ and calculations are taken modulo $k$, and $G_\infty=(a_0,a_1,a_2,...,a_{k-1})$.\\
	We now construct the partial $C_k$-factors of $K_{k+1} \times K_t$ as follows:\\
	Set $A_i=\{i\} \times \mathbb{Z}_t,$ for every $i \in \mathbb{Z}_{k} \cup \{\infty\}.$ Also, let $A_0,A_1,...,A_{k-1}, A_\infty$ be the partite sets of $K_{k+1} \times K_t.$\\
	(i) For odd  $i<k$, let $G_i^r= F_r(A_i,A_{i+1}) \oplus F_{t-r}(A_{i+1},A_{i-1}) \oplus... \oplus F_{t-r}(A_\infty,A_i)$, \\
	(ii) For even $i < k$, let $G^r_i= F_{t-r}(A_i,A_{i+1}) \oplus F_{r}(A_{i+1},A_{i-1}) \oplus...\oplus  F_{r}(A_\infty,A_i)$ and \\
	(iii) let $G^r_{\infty}= F_r(A_0,A_1) \oplus F_{t-r}(A_1,A_2) \oplus...\oplus F_{t-r}(A_{k-1},A_0)$,\\ where each $G_i^r$ and $G_\infty^r$ misses $A_{i+\frac{k}{2}}$ and $A_\infty$ respectively, $ 0 \leq i \leq k-1, \ 1 \leq r \leq \frac{t-1}{2}$
	and the additions are taken modulo $k$.\\
	One can check that each $G^r_i,\ i \in \mathbb{Z}_k \cup \{\infty\},$ $ 1 \leq r \leq \frac{t-1}{2}$ is a partial $C_k$-factor of  $K_{k+1} \times K_t$ and hence  $K_{k+1} \times K_t$ has a partial $C_k$-factorization.
\end{proof}
\begin{lem} \label{lem}
There exists a $C_{kt}$-factorization of $C_k \times K_t,$ for all even $k \geq 4,$  and odd $t \geq 3$.
\end{lem}
\begin{proof}
We can write
\begin{eqnarray}
C_k \times K_t & \cong& C_k \times \{C_t^1 \oplus...\oplus C_t ^{\frac{t-1}{2}}\} \nonumber\\
& \cong & \oplus_{i=1}^{\frac{t-1}{2}}(C_k \times C_t^i) \nonumber\\
& \cong & \oplus_{i=1}^{\frac{t-1}{2}}(C_t^i \times C_k) \nonumber
\end{eqnarray}
We know that $C_t^i \times C_k \cong C_t \times C_k$ and find its $C_{kt}$-factors as follows:\\
Let $V(C_t)=\{a_0,a_1,...,a_{t-1}\}$ then $V(C_t \times C_{k})=\{A_i| 0 \leq i \leq t-1\},$ where $A_i=a_i \times \mathbb{Z}_{k} =\{a_i^0,a_i^1,...,a_i^{k-1}\}.$ Let
\begin{eqnarray}
\mbox{(i)}  \ \hspace{1cm} {G}&=&\{ \oplus_{i=0}^{\frac{t-1}{2}} F_1(A_{2i},A_{2i+1})\} \oplus \{\oplus_{i=0}^{\frac{t-3}{2}} F_{k-1}(A_{2i+1},A_{2i+2})\}  \ \  \mbox{and} \label{G}\\
	\mbox{(ii)}  \hspace{1cm} {H}&=& \{\oplus_{i=0}^{\frac{t-1}{2}} F_{k-1}(A_{2i},A_{2i+1})\} \oplus \{\oplus_{i=0}^{\frac{t-3}{2}} F_{1}(A_{2i+1},A_{2i+2})\}, \label{H}
\end{eqnarray}

where the subscripts of $A$ are taken modulo $t$. One can check that both $\mathcal{G}$ and $\mathcal{H}$ are the $C_{kt}$-factors of $C_t \times C_k$. Thus each $C_t^i \times C_k$ has $2$ $C_{kt}$-factors and hence together gives a $C_{kt}$-factorization of $C_k \times K_t.$
\end{proof}
\begin{thm}\label{I2}
	There exists a partial $C_{kt}$-factorization of $K_{k+1} \times K_t,$ for all even $k \geq 4,$  and odd $t \geq 3$.
\end{thm}
\begin{proof}
By Theorem  \ref{t4}, $K_{k+1}(2)$ has a near $C_k$-factorization. The required near $C_k$-factors are given by\\
	$G _i=(a_i,a_{i+1},a_{i-1},a_{i+2},a_{i-2},...,a_{i+\frac{k}{2}-1},a_{i-\frac{k}{2}+1},a_{\infty}),$ $0 \leq i \leq k-1$ and calculations are taken modulo $k$, and $G_\infty=(a_0,a_1,a_2,...,a_{k-1})$.	We can write 
	\begin{eqnarray}
	K_{k+1} (2) \times K_t & \cong&  \oplus_{i=0}^{k-1} (G_i \oplus G_ \infty) \times K_t \ \mbox{by Theorem \ref{t4} } \nonumber\\
	& \cong&  \oplus_{i=0}^{k-1} (G_i \times K_t) \oplus (G_ \infty \times K_t)
	\end{eqnarray}
By (\ref{G}) of Lemma \ref{lem}, $G_i \times K_t \cong C_k \times K_t$  has a $C_{kt}$-factorization $\mathcal{G}_i=\{G_{i,1},G_{i,2},...,G_{i, \frac{t-1}{2}}\}$, for every $ i \in \mathbb{Z}_k$.\\
Considering that $G_\infty \times K_t \cong C_k \times K_t$ and by   (\ref{H}) of Lemma \ref{lem},  $G_\infty \times K_t$ has a $C_{kt}$-factorization $\mathcal{G}_\infty=\{H^{\infty}_{1},H^{\infty}_{2},...,H^{\infty}_{ \frac{t-1}{2}}\}$.\\
One can check that $\mathcal{G}=\{G_{i,j} \oplus H^{\infty}_{j} | \ i \in \mathbb{Z}_k, \ 1 \leq j \leq \frac{t-1}{2}\}$ gives a partial $C_{kt}$-factorization of $K_{k+1} \times K_t$.  
\end{proof}
\begin{thm}\label{nr}
	There exists a $C_k$-factorization of $K_3 \times K_{ky}$ for all even $k \geq 6$ and all odd $y\geq 1$.
\end{thm}
\begin{proof}
We establish the proof in two cases\\
{\bf Case (i): $y=1$}
\begin{eqnarray}
	K_3 \times K_k &\cong& K_k \times K_3 \label{eq2}\\
	&\cong&(  \underbrace{C_k\oplus ...\oplus C_k}_{\frac{k}{2}-2 \ \  \mbox{times}} \oplus G ) \times K_3, \ \mbox{by walecki's contruction} \nonumber\\
	&\cong& (C_k \times K_3) \oplus...\oplus (C_k \times K_3)\oplus (G \times K_3),\nonumber
	\end{eqnarray}
	where $G$ is a cubic graph having perfect $1$-factorization as defined in \cite{PP}.\\
	One can check that each $C_k \times K_3$ has $2$ $C_k$-factors by Theorem \ref{l1} and  $G \times K_3$ has  $3$ $C_k$-factors by Theorem \ref{f}. Hence $K_3 \times K_k$ is factorable into $(k-1) \ C_k$-factors. \\
{\bf Case (ii): $y \geq 3$ is odd}\\
	We can write
	\begin{eqnarray}\label{eq1}
	K_3 \times K_{ky} \cong \{(K_3 \times K_y) \otimes \bar{K}_k\} \oplus y (K_3 \times K_k)
	\end{eqnarray}
	Graphs in the R.H.S of (\ref{eq1}) can be obtained by making $y$ holes of type $K_3 \times K_k$ in $K_3 \times K_{ky}$ and identifying each $k$-subsets of $K_{ky}$ (in the resulting graph) into a single vertex and two of them are adjacent if the corresponding $k$-subsets form a $K_{k,k}$ in $K_3 \times K_{ky}.$ By expanding the vertices into $k$-subsets we get the first graph of  (\ref{eq1})\\
	Now, we construct the  $C_k$-factors of $(K_3 \times K_y) \otimes \bar{K}_k$   as follows:
	\begin{eqnarray}
	(K_3 \times K_y) \otimes \bar{K}_k&\cong& (C_3^1 \oplus C_3^2 \oplus...\oplus C_3^{y-1}) \otimes \bar{K}_k, \ \mbox{by Theorem \ref{l1}} \nonumber\\
	&\cong& (C_3^1 \otimes \bar{K}_k )\oplus...\oplus (C_3^{y-1} \otimes \bar{K}_k), \nonumber
	\end{eqnarray}
	each $C_3^i, \ 1\leq i \leq y-1, $ is a $C_3$-factor of $K_3 \times K_y$. One can observe that $C_3^i \otimes \bar{K}_k \cong y(C_3 \otimes \bar{K}_k) \cong y (K_{k,k,k})$. By Theorem \ref{ttt}, each $K_{k,k,k}$ is factorable into $k \ C_k$-factors. Thus we get $k(y-1)$ $C_k$-factors of $(K_3 \times K_y) \otimes \bar{K}_k$. By (\ref{eq2}), each $K_3 \times K_k$ is factorable into $(k-1) \ C_k$-factors. 
	Finally, the two graph in   (\ref{eq1}), together gives $(ky-1) \ C_k$-factors of $K_3 \times K_{ky}$. Hence, $C_k$-factorization of  $K_3 \times K_{ky}$ exists.
\end{proof}
\begin{thm}\label{nr}
	There exists a $C_k$-factorization of $K_3 \times K_{ky}$ for all even $k \geq 6$ and all odd $y\geq 1$.
\end{thm}
\begin{proof}
We establish the proof in two cases\\
{\bf Case (i): $y=1$}
\begin{eqnarray}
	K_3 \times K_k &\cong& K_k \times K_3 \label{eq2}\\
	&\cong&(  \underbrace{C_k\oplus ...\oplus C_k}_{\frac{k}{2}-2 \ \  \mbox{times}} \oplus G ) \times K_3, \ \mbox{by walecki's contruction} \nonumber\\
	&\cong& (C_k \times K_3) \oplus...\oplus (C_k \times K_3)\oplus (G \times K_3),\nonumber
	\end{eqnarray}
	where $G$ is a cubic graph having perfect $1$-factorization as defined in \cite{PP}.\\
	One can check that each $C_k \times K_3$ has $2$ $C_k$-factors by Theorem \ref{l1} and  $G \times K_3$ has  $3$ $C_k$-factors by Theorem \ref{f}. Hence $K_3 \times K_k$ is factorable into $(k-1) \ C_k$-factors. \\
{\bf Case (ii): $y \geq 3$ is odd}\\
	We can write
	\begin{eqnarray}\label{eq1}
	K_3 \times K_{ky} \cong \{(K_3 \times K_y) \otimes \bar{K}_k\} \oplus y (K_3 \times K_k)
	\end{eqnarray}
	Graphs in the R.H.S of (\ref{eq1}) can be obtained by making $y$ holes of type $K_3 \times K_k$ in $K_3 \times K_{ky}$ and identifying each $k$-subsets of $K_{ky}$ (in the resulting graph) into a single vertex and two of them are adjacent if the corresponding $k$-subsets form a $K_{k,k}$ in $K_3 \times K_{ky}.$ By expanding the vertices into $k$-subsets we get the first graph of  (\ref{eq1})\\
	Now, we construct the  $C_k$-factors of $(K_3 \times K_y) \otimes \bar{K}_k$   as follows:
	\begin{eqnarray}
	(K_3 \times K_y) \otimes \bar{K}_k&\cong& (C_3^1 \oplus C_3^2 \oplus...\oplus C_3^{y-1}) \otimes \bar{K}_k, \ \mbox{by Theorem \ref{l1}} \nonumber\\
	&\cong& (C_3^1 \otimes \bar{K}_k )\oplus...\oplus (C_3^{y-1} \otimes \bar{K}_k), \nonumber
	\end{eqnarray}
	each $C_3^i, \ 1\leq i \leq y-1, $ is a $C_3$-factor of $K_3 \times K_y$. One can observe that $C_3^i \otimes \bar{K}_k \cong y(C_3 \otimes \bar{K}_k) \cong y (K_{k,k,k})$. By Theorem \ref{ttt}, each $K_{k,k,k}$ is factorable into $k \ C_k$-factors. Thus we get $k(y-1)$ $C_k$-factors of $(K_3 \times K_y) \otimes \bar{K}_k$. By (\ref{eq2}), each $K_3 \times K_k$ is factorable into $(k-1) \ C_k$-factors. 
	Finally, the two graph in   (\ref{eq1}), together gives $(ky-1) \ C_k$-factors of $K_3 \times K_{ky}$. Hence, $C_k$-factorization of  $K_3 \times K_{ky}$ exists.
\end{proof}
\begin{thm} \label{2=6}
	There exists a $C_k$-factorization of $K_3 \times K_{ky}$ for all even integers $k > 6$ and $y \geq 2$. 
\end{thm}
\begin{proof}
	We can write
	\begin{eqnarray}\label{eq4}
\hspace{-.8cm}	K_3 \times K_{ky} &\cong& \{(K_3 \times K_y) \otimes \bar{K}_k\} \oplus y (K_3 \times K_k) \ \mbox{(by case (ii) of  Theorem \ref{nr})}\\
	&\cong&\{(C_ {3y}^1 \oplus...\oplus C_{3y}^{y-1}) \otimes \bar{K}_k\} \oplus y (K_3 \times K_k), \ \mbox{by Theorem \ref{hc}} \nonumber
	\end{eqnarray}
	Since each $C_{3y}^i \cong (F_1\oplus F_2), \ 1 \leq i \leq y-1$, where $F_j, \ j=1,2$ is a $1$-factor of $C_{3y}$. We know that $F_j \otimes \bar{K}_k \cong \frac{3y}{2}(K_2 \otimes \bar{K}_k) \cong \frac{3y}{2}(K_{k,k})$ and by Theorem \ref{mn}, each $K_{k,k}$ has $\frac{k}{2} \ C_k$-factors . Hence $(K_3 \times K_y) \otimes \bar{K}_k$ is factorable into $(y-1)k \ C_k$-factors. By (\ref{eq2}), $ K_3 \times K_k$ is factorable into $(k-1) \ C_k$-factors. Finally, the two graphs in  (\ref{eq4}), together gives $(ky-1) \ C_k$-factors of $K_3 \times K_{ky}$. Hence, $C_k$-factorization of  $K_3 \times K_{ky}$ exists.
\end{proof}
\begin{thm} \label{3.7}
	There exists a $C_{kt}$-factorization of $C_k \times K_s,$ for all even $k \geq 4, \ t \geq 3$ and $s \equiv 0(mod \ 2t).$
\end{thm}
\begin{proof}
	We can write
	\begin{eqnarray}
	C_k \times K_s & \cong & C_k \times K_{2ty} \nonumber\\
	& \cong & C_k \times K_{t(2y)} \nonumber
	\end{eqnarray}
	{\bf Case (i): $t \geq 3$ is odd }\\
	By Theorem \ref{l1}, $C_k \times K_{t(2y)}$ has $2ty-1$ $C_{kt}$-factors.\\
	{\bf Case (ii): $t \geq 4$ is even }\\
	We write
	\begin{eqnarray}
	C_k \times K_{t(2y)} & \cong & \{(C_k \times K_{2y}) \otimes \bar{K}_t\} \oplus 2y(C_k \times K_t) \label{3.7eq}
	\end{eqnarray}
	Graphs in the R.H.S of (\ref{3.7eq}) can be obtained by making $2y$ holes of type $C_k \times K_t$ in $C_k \times K_{t(2y)}$ and identifying each $t$-subsets of $K_{t(2y)}$ (in the resulting graph)  into a single vertex and two of them are adjacent if the corresponding $t$- subsets in $C_k \times K_{t(2y)}$ form a $K_{t,t}.$ By expanding the vertices into $t$-subsets we get the first graph in  (\ref{3.7eq}).\\
	Now, we construct the $C_{k}$-factors of $C_k \times K_{2y}$ as follows:\\
	Let $V(C_k)=\{a_0,a_1,...,a_{k-1}\}$ then $V(C_k \times K_{2y})=\{A_i| 0 \leq i \leq k-1\},$ where $A_i=a_i \times \mathbb{Z}_{2y}=\{a_i^0,a_i^1,...,a_i^{2y-1}\}.$
	\begin{itemize}
		\item[] ${G}_l=\oplus_{j=0}^{\frac{k}{2}-1}\{F_l(A_{2j},A_{2j+1})\} \oplus \{\oplus _{j=0}^{\frac{k}{2}-1}\{F_{2y-l}(A_{2j+1},A_{2j+2})\}\},$
	\end{itemize}
	where the subscripts of $A$ are taken modulo $k.$ One can check that each $G_l$, $1 \leq l \leq 2y-1,$ is a $C_k$-factor of $C_k \times K_{2y}.$ We know that $(C_k \times K_{2y}) \otimes \bar{K}_t \cong (2y-1)\{2y(C_k \otimes \bar{K}_t)\}$ and by Theorem \ref{hc2}, each $C_k \otimes \bar{K}_t$ has $t$ $C_{kt}$-factors. Thus we have obtained $(2y-1)t$ $ C_{kt}$-factors of $(C_k \times K_{2y}) \otimes \bar{K}_t.$ Further, by Theorem \ref{pten}, each $C_k \times K_t$ has $t-1$ $C_{kt}$-factors. The combination of the $2ty-1$ $C_{kt}$-factors of (\ref{3.7eq}) gives a $C_{kt}$-factors of $C_k \times K_{s}$. Hence $C_{kt}$-factors of  $C_k \times K_{s}$ exists. 
\end{proof}
\section{$k$-ARCS of $(K_u \times K_g)(\lambda)$}
In this section, we investigate the existence of  $k$-ARCS of tensor product of complete graphs with edge multiplicity $\lambda.$

\begin{thm}\label{tt2}
	For all even $k \geq 4$, $u \equiv 1 (mod \ k)$ and $g \geq 2$, there exists a $k$-ARCS of $(K_u \times {K}_{g})(2)$.
	\begin{proof}
		Let $u=kx+1$, where $x \geq 1$.\\ By Theorem \ref{t4}, let $\mathcal{C}=\{\mathcal{C}_k^1,\mathcal{C}_k^2, \mathcal{C}_k^3,..., \mathcal{C}_k^u \}$ be the near $C_k$-factorization of $K_u(2)$, where each
		$\mathcal{C}_k^i$, $1 \leq i \leq u$ is a near $C_k$-factor of $K_u(2)$.
		Now we write
		\begin{eqnarray}
		(K_u \times K_g)(2) &\cong& K_u(2) \times K_g \nonumber \\
		&\cong& (\mathcal{C}_k^1 \oplus \mathcal{C}_k^2 \oplus... \oplus \mathcal{C}_k^u) \times K_{g}, \ \ \mbox{by Theorem \ref{t4}} \nonumber \\
		&\cong& (\mathcal{C}_k^1 \times K_{g}) \oplus (\mathcal{C}_k^2 \times K_{g}) \oplus ... \oplus(\mathcal{C}_k^u \times K_{g}). \nonumber 
		\end{eqnarray}
		Since each $ \mathcal{C}_k^i \times K_{g}\cong \frac{u-1}{k}(C_k \times K_{g}) $,
		by Theorem \ref{l1}, each $C_k \times K_{g}$ has $g-1$ $C_k$-factors. The $C_k$-factors obtained above become the partial $C_k$-factors of  $(K_u \times K_g)(2).$ 
		Thus, in total we get $u(g-1)$ partial $C_k$-factors of  $(K_u \times K_g)(2).$ 
		Therefore $k$-ARCS of  $(K_u \times {K}_{g})(2)$ exists.
	\end{proof}
\end{thm}
\begin{thm}\label{tt3}
	For all even $k \equiv 0(mod \ 4), \ x \in \mathbb{N} \setminus \{2\}$, $u=kx+1$ and  $g \geq 3$ is odd , there exists a $k$-ARCS of $K_u \times {K}_{g}.$
\end{thm}
\begin{proof}
Let $u=kx+1$, $g \geq 3$ is odd and $k=4t$, where $x \in \mathbb{N} \setminus  \{2\},$ and $y,t \geq 1$.\\
We establish the proof in two cases.\\
{\bf Case (i): $x=1$}\\
Let $A_0, A_1,...,A_{4t-1},A_{\infty}$ be the parts of  $K_{4t+1} \times K_g.$
By Theorem \ref{I1}, $ K_{4t+1} \times K_g$ has a partial $C_k$-factorization $\mathcal{G}=\{G_{s,h} \oplus G_{h}^{\infty}| s\in \mathbb{Z}_{4t},   1 \leq h \leq \frac{g-1}{2}\}$ where
 each $G_{s,h}$ is a  partial $C_k$-factor missing the partite set $A_{s+2t},$  and the subscripts of $A$ are taken modulo $4t$; 
 each $G_{h}^{\infty}$ is a partial $C_k$- factor missing the partite set $A_{\infty}.$\\
Hence $k$-ARCS of  $K_{4t+1} \times {K}_{g}$ exists.\\
{\bf Case (ii): $x >2$}\\
Set $X_i=\{i\} \times \mathbb{Z}_{4t}$ and let $\Gamma_i$ be the complete graph of order $4t+1$ with vertex-set $V(\Gamma_i)= X_i \cup \{\infty\},$ for every $i \in \mathbb{Z}_x.$ Also, let $X_0,X_1,...,X_{x-1}$ be the parts of the complete multipartite graph $K_x \otimes \bar{K}_{4t},$ and set
\begin{eqnarray}
K_{4tx+1} \cong (K_x \otimes \bar{K}_{4t}) \oplus (\oplus_{i=0}^{x-1}\Gamma_i) \label{R1}
\end{eqnarray}
Considering that $K_x \otimes \bar{K}_{4t}\cong (K_x \otimes \bar{K}_{2}) \otimes \bar{K}_{2t},$ and by Theorem \ref{k2}, $K_x \otimes \bar{K}_{2}$ has a partial $1$-factorization $\{\Delta_0,\Delta_1,...,\Delta_{x-1}\}$ . We have that
\begin{eqnarray}
K_x \otimes \bar{K}_{4t} & \cong& (\Delta_0 \oplus ... \oplus \Delta_{x-1}) \otimes \bar{K}_{2t} \nonumber\\
&\cong& (\Delta_0 \otimes \bar{K}_{2t}) \oplus...\oplus (\Delta_{x-1} \otimes \bar{K}_{2t}), \nonumber
\end{eqnarray} 
where each $\Gamma'_i=\Delta_i \otimes \bar{K}_{2t}$ misses $X_i.$ Therefore, by (\ref{R1}) it follows that
\begin{eqnarray}
K_{4tx+1} \cong \oplus_{i=0}^{x-1}(\Gamma'_i \oplus \Gamma_i) \nonumber
\end{eqnarray}
where $V(\Gamma'_i \oplus \Gamma_i)=(\mathbb{Z}_x \times \mathbb{Z}_{4t}) \cup \{\infty\}.$ Also,
\begin{eqnarray}
K_{4tx+1} \times K_g &\cong& \oplus_{i=0}^{x-1}(\Gamma'_i \oplus \Gamma_i) \times K_g \nonumber\\
&\cong& \oplus_{i=0}^{x-1}\{(\Gamma'_i \times K_g) \oplus (\Gamma_i \times K_g)\} . \label{R2}
\end{eqnarray}
\par By Theorem \ref{mn}, $\Gamma'_i$ has a $C_{4t}$-factorization $\mathcal{G}'_i=\{G'_{i,1},G'_{i,2},...,G'_{i,2t}\},$ and by Theorem \ref{l1}, $G'_{i,j} \times K_g$ has a $C_k$-factorization $\mathcal{G}'_{i,j}=\{G'_{i,j,1},G'_{i,j,2},...,G'_{i,j,g-1}\},$ for every $i \in \mathbb{Z}_x,$ and $ j \in \{1,2,...,2t\}.$ Hence, $\oplus_j \mathcal{G}'_{i,j}$ gives a $C_k$-factorization of $\Gamma'_i \times K_g.$
\par Now consider \ $\Gamma_i \times K_g \cong K_{4t+1} \times K_g $ and let $A_0, A_1,...,A_{4t-1},A_{\infty}$ be the parts of  $K_{4t+1} \times K_g.$
By Theorem \ref{I1}, $ K_{4t+1} \times K_g$ has a partial $C_k$-factorization 
$\mathcal{G}_i=\{G_{i,s,h} \oplus G_{i,h}^{\infty}| s\in \mathbb{Z}_{4t},   1 \leq h \leq \frac{g-1}{2}\}$ where 
\ each $G_{i,s,h}$ is a partial $C_k$-factor missing the partite set $A_{s+2t},$  and the subscripts of $A$ are taken modulo $4t$; 
each $G_{i,h}^{\infty}$ is a partial $C_k$-factor missing the partite set $A_{\infty}.$\\
One can check that $\mathcal{G}=\{G'_{i,j,l} | i \in \mathbb{Z}_x, 1 \leq j \leq 2t, 1 \leq l \leq g-1\} \cup \{G_{i,s,h} \oplus G_{i,h}^{\infty}| i \in \mathbb{Z}_x, s\in \mathbb{Z}_{4t},   1 \leq h \leq \frac{g-1}{2}\}$
 gives a partial $C_k$-factorization of $K_{4tx+1} \times K_g.$ 
Therefore $k$-ARCS of  $K_u \times {K}_{g}$ exists.
\end{proof}
\begin{thm}\label{tt1}
	For all even $k \geq 4$, odd $u \geq 3$	and $g \equiv 0 (mod \ k)$, there exists a $k$-ARCS of $(K_u \times {{K}_g})(2)$.
	\begin{proof}
		Let $u=2x+1$ and $g=ky$, where $x,y \geq 1$.\\
		By Theorem \ref{t1}, let $\mathcal{F}=\{F_1,F_2,....,F_u\}$ be the near $1$-factorization of $K_u$, where each $F_i$, $1 \leq i \leq u$ is a near $1$-factor of $K_u$.
		We can write

		\begin{eqnarray}
		(K_u \times K_g)(2) &\cong& \{(F_1 \oplus F_2 \oplus .... \oplus F_u) \times K_g\} (2), \ \ \mbox{by  Theorem  \ref{t1}} \nonumber \\
		& \cong & (F_1 \times K_g) (2) \oplus .... \oplus (F_u \times K_g)(2) \nonumber \\
		& \cong & (F_1 \times K_{ky}) (2) \oplus .... \oplus (F_u \times K_{ky})(2) \nonumber
		\end{eqnarray}
		Since $(F_i \times K_{ky})(2) \cong F_i \times K_{ky}(2)  \cong  \frac{u-1}{2} \ (K_2 \times K_{ky}(2))$ and by Theorem \ref{t6}, let $\mathcal{C}=\{\mathcal{C}_k^1,\mathcal{C}_k^2,...,\mathcal{C}_k^{ky-1}\}$ be the $C_k$-factorization of $K_{ky}(2),$  where each $C_k^j, \ 1 \leq j \leq ky-1$ is a $C_k$-factor of $K_{ky}(2)$. One can observe that $\mathcal{C}_k^j \cong y C_k$ and  $K_2 \times C_k \cong C_k \times K_2.$  By Theorem \ref{l1}, each $C_k \times K_2$ is a $C_k$-factor. Thus, in total we get $(2x+1)(ky-1)$ partial $C_k$-factors of $(K_u \times K_g)(2)$. Therefore $k$-ARCS of  $(K_u \times {K}_{g})(2)$ exists.
	\end{proof}
\end{thm}

\begin{thm} \label{tt4}
	For all even  $k \geq 6$, $ y \geq  1$ is odd,  $x \in \mathbb{N} \setminus \{2\}, \ u=4x $ and $g=ky,$	
	there exists a $k$-ARCS of $(K_u \times {K}_g)(2).$
\end{thm} 
\begin{proof}
Let $ u =4x$ and  $g=ky$, where  $x \in \mathbb{N} \setminus \{ 2\}, \ y \geq 1$ is odd.\\
We establish the proof in two cases.\\
{\bf Case (i): $x=1.$}\\
We can write $$(K_4 \times K_{ky})(2) \cong K_4(2) \times K_{ky}$$
  By Theorem \ref{t5}, $ K_4(2)$ has a near $C_3$-factorization $\mathcal{G}=\{G_{1},G_{2},G_{3},G_{4}\}$, and by Theorem \ref{nr}, each $G_{i} \times K_{ky}, \ 1 \leq i \leq 4$   has a  $C_k$-factorization $\mathcal{G}_{i}=\{G_{i,1},G_{i,2},...,G_{i,ky-1}\}.$  Hence $ (K_4\times K_{ky})(2)$ has a partial $C_k$-factorization. \\
{\bf Case (ii): $x >2.$}\\
Set $X_i=\{i\} \times \mathbb{Z}_{4}$ and let $\Gamma_i$ be the complete graph of order $4$ with vertex-set $V(\Gamma_i)= X_i ,$ for every $i \in \mathbb{Z}_x.$ Also, let $X_0,X_1,...,X_{x-1}$ be the parts of the complete multipartite graph $K_x \otimes \bar{K}_{4},$ and set
\begin{eqnarray}
K_{4x} \cong (K_x \otimes \bar{K}_{4}) \oplus (\oplus_{i=0}^{x-1}\Gamma_i) \label{R4}
\end{eqnarray}
By Theorem \ref{k2}, $K_x \otimes \bar{K}_{4}$ has a partial $1$-factorization $\{\Gamma'_0,\Gamma'_1,...,\Gamma'_{x-1}\}$  where each $\Gamma'_i$ misses $X_i.$ Therefore, by (\ref{R4}) it follows that
\begin{eqnarray}
K_{4x} = \oplus_{i=0}^{x-1} (\Gamma'_i \oplus \Gamma_i), \nonumber
\end{eqnarray} 	
where	$V(\Gamma'_i \oplus \Gamma_i)=V(K_{4x})=\mathbb{Z}_x \times \mathbb{Z}_4,$ Also
\begin{eqnarray}
(K_{4x} \times K_{ky}) (2) \cong \oplus_{i=0}^{x-1} \Big\{ \big( \Gamma'_i  \times K_{ky}(2)\big) \oplus \big(\Gamma_i \times K_{ky}(2)\big) \Big\} . \nonumber
\end{eqnarray}
 By Theorem \ref{t6} and   Theorem \ref{l1}, $\Gamma'_i \times K_{ky}(2)$ has a $C_k$-factorization $\mathcal{G}'_i=\{G'_{i,1},G'_{i,2},...,G'_{i,(ky-1)}\},$ for every $ i \in \mathbb{Z}_x.$	\\
 It is clear that $\Gamma_i \times K_{ky}(2) \cong K_4
 (2) \times K_{ky}$. By Theorem \ref{t5}, $ K_4(2)$ has a near $C_3$-factorization $\mathcal{G}_i=\{G_{i,1},G_{i,2},G_{i,3},G_{i,4}\}$, and by Theorem \ref{nr}, $G_{i,j} \times K_{ky}$  has a  $C_k$-factorization $\mathcal{G}_{i,j}=\{G_{i,j,1},G_{i,j,2},...,$ $G_{i,j,ky-1}\}.$  Hence $\Gamma_i \times K_{ky}(2)$ has a partial $C_k$-factorization. \\
One can check that $\mathcal{G}=\{G'_{i,l} \oplus G_{i,j,l}| i \in \mathbb{Z}_x, \ 1\leq j \leq 4, \ 1 \leq l \leq ky-1\}$ gives a partial $C_k$-factorization of $(K_{4x} \times K_{ky})(2).$
 Therefore $k$-ARCS of  $(K_u \times {K}_{g})(2)$ exists.
\end{proof}
\begin{thm} \label{26}
	For all even integers $k > 6$, $y \geq 2$, $u=4x, \ x \in \mathbb{N} \setminus \{2\}$ and $g=ky,$	
	there exists a $k$-ARCS of $(K_u \times {K}_g)(2)$.
\end{thm}
\begin{proof}
Let $ u =4x$ and  $g=ky$, where  $x \in \mathbb{N} \setminus \{2\}, \ y \geq 2$ is even.\\
We establish the proof in two cases.\\
{\bf Case (i): $x=1.$}\\
We can write $$(K_4 \times K_{ky})(2) \cong K_4(2) \times K_{ky}$$
  By Theorem \ref{t5}, $ K_4(2)$ has a near $C_3$-factorization $\mathcal{G}=\{G_{1},G_{2},G_{3},G_{4}\}$, and by Theorem \ref{2=6}, each $G_{i} \times K_{ky}, \ 1 \leq i \leq 4$   has a  $C_k$-factorization $\mathcal{G}_{i}=\{G_{i,1},G_{i,2},...,G_{i,ky-1}\}.$  Hence $ (K_4\times K_{ky})(2)$ has a partial $C_k$-factorization. \\
{\bf Case (ii): $x >2.$}\\
Using a similar procedure as in case (ii) of Theorem \ref{tt4}, we can get the following
\begin{eqnarray}
(K_{4x} \times K_{ky}) (2) \cong \oplus_{i=0}^{x-1} \Big\{ \big( \Gamma'_i  \times K_{ky}(2)\big) \oplus \big(\Gamma_i \times K_{ky}(2)\big) \Big\}. \nonumber
\end{eqnarray}
By Theorem \ref{t6} and Theorem \ref{l1}, $\Gamma'_i \times K_{ky}(2)$ has a $C_k$-factorization $\mathcal{G}'_i=\{G'_{i,1},G'_{i,2},...,G'_{i,(ky-1)}\},$ for every $ i \in \mathbb{Z}_x.$	\\
	It is clear that $\Gamma_i \times K_{ky}(2) \cong K_4
	(2) \times K_{ky}$. By Theorem \ref{t5}, $ K_4(2)$ has a partial $C_3$-factorization $\mathcal{G}_i=\{G_{i,1},G_{i,2},G_{i,3},G_{i,4}\}$, and by Theorem \ref{2=6}, $G_{i,j} \times K_{ky}$  has a  $C_k$-factorization $\mathcal{G}_{i,j}=\{G_{i,j,1},G_{i,j,2},...,G_{i,j,ky-1}\}.$  Hence $\Gamma_i \times K_{ky}(2)$ has a partial $C_k$-factorization. \\
	One can check that $\mathcal{G}=\{G'_{i,l} \oplus G_{i,j,l}| i \in \mathbb{Z}_x, \ 1\leq j \leq 4, \ 1 \leq l \leq ky-1\}$ gives a partial $C_k$-factorization of $(K_{4x} \times K_{ky})(2).$
	Therefore $k$-ARCS of  $(K_u \times {K}_{g})(2)$ exists.
\end{proof}
\begin{thm} \label{pt1}
	Let $k=p_1p_2p_3...p_t \geq 6$ be even, where $p_1,p_2,p_3,...,p_t$ are primes not necessarily distinct. If $u \equiv 1 (mod \ p_1p_2p_3...p_i )$, $g \equiv 0 (mod \ 2 p_{i+1}p_{i+2}p_{i+3}...p_t)$, $1 \leq i < t-1,$  at least one of  the $p_1,p_2,p_3,...,p_i$ is $2$ and  $p_{i+1},p_{i+2},...,p_t$ are odd primes, then there exists a $k$-ARCS of $(K_u \times {K}_g)(2)$.
	\begin{proof}
		We establish the proof in two cases\\
		{\bf Case(i)}  $i=1$ and $p_1=2$.\\
		Let  $u=2x+1$, $g=2sy$, where $x, y \geq 1,$ $s=p_2p_3...p_t$ and $k=2s.$  \\
		Now
		\begin{eqnarray}
		(K_u \times K_g)(2) &\cong&\{(F_1 \oplus F_2 \oplus ... \oplus F_u) \times K_g\} (2), \ \mbox{by Theorem \ref{t1}} \nonumber \\
		&\cong&(F_1 \times K_{2sy})(2) \oplus...\oplus (F_u \times K_{2sy})(2), \nonumber
		\end{eqnarray}
		where each $F_j$, $1 \leq j \leq u$ is a near $1$-factor of $K_u$.
		Since $(F_j \times K_{2sy})(2) \cong \frac{u-1}{2}( K_2 \times K_{ky}(2))$ and
		by Theorem \ref{t6}, $ K_{ky}(2)$ has a $C_k$-factorization $\{{C}_k^1,{C}_k^2,...,{C}_k^{ky-1}\},$   where $C_k^l, \ \ 1 \leq l \leq ky-1,$ is a $C_k$-factor of $K_{ky}(2)$.
		By Theorem \ref{l1}, each $C_k \times K_2$ is a $C_k$-factor. Thus we have obtained  $ky-1$ $C_k$-factors of $ K_{ky}(2) \times K_2$. The $C_k$-factors obtained above become the partial $C_k$-factors of $(K_u \times K_g)(2)$. Thus we get $(2x+1)(ky-1)$ partial $C_k$-factors of $(K_u \times K_g)(2)$. Hence, $k$-ARCS of $(K_u \times {K}_{g})(2)$ exists.\\
		{\bf Case (ii)} $1<i\ < t-1$ \\
		Suppose that 	$r=p_1p_2...p_i$ and $s=p_{i+1}p_{i+2}...p_t,$ $1 < i < t-1$\\
		Let $u=rx+1$, $g=2sy$, where $ x,y \geq 1$ and $k=rs.$\\
		We can write
		\begin{eqnarray}
		(K_u \times K_g)(2)&\cong&K_u(2) \times K_g \nonumber \\
		&\cong&(\mathcal{C}_{r}^1 \oplus ... \oplus \mathcal{C}_{r}^u) \times K_g, \mbox{by Theorem \ref{t4}} \nonumber \\
		&\cong&(\mathcal{C}_{r}^1 \times K_{(2s)y} ) \oplus ... \oplus (\mathcal{C}_{r}^u \times K_{(2s)y} ), \nonumber
		\end{eqnarray}
		where each $\mathcal{C}_{r}^j$, $1 \leq j \leq u$ is a near $C_{r}$-factor of $K_u(2)$.
		Since $ \mathcal{C}_r^j \times K_{(2s)y}
		\cong x (C_r \times K_{s(2y)})$,
		by Theorem \ref{3.7} each $C_r \times K_{s(2y)}$ has a $C_{rs}$-factorization. Thus we have obtained $2sy-1 \ C_k$-factors of $C_r \times K_{s(2y)}$. The $C_k$-factors obtained above become the partial $C_k$-factors of $(K_u \times K_g)(2)$. Thus, in total we get $(rx+1)(2sy-1)$ partial $C_k$-factors of $(K_u \times K_g)(2)$. Therefore $k$-ARCS of  $(K_u \times {K}_{g})(2)$ exists.
	\end{proof}
\end{thm}

\begin{thm} \label{pt13}
	Let $k=p_1p_2p_3...p_t\geq 6$ be even, where $p_1,p_2,...,p_t$ are primes not necessarily distinct. If $u \equiv 1(mod \ p_1p_2p_3...p_i)$, $g \equiv (p_{i+1}p_{i+2}p_{i+3}...p_t) (mod \ 2 p_{i+1}p_{i+2}p_{i+3}...p_t)$ is odd, $1 \leq i \leq t-1$ and at least one of  the $p_1,p_2,p_3,...,p_i$ is $2$, then there exists a $k$-ARCS of $(K_u \times {K}_{g})(2)$.
\end{thm}
\begin{proof}
	Suppose that $r=p_1p_2...p_i$ and $s=p_{i+1}p_{i+1}...p_t$, $1 \leq i \leq t-1$.\\
	We establish the proof in two cases.\\
	{\bf Case (i): $i=1, \ p_1=2$}\\
	Let $u=2x+1$, $g=s(2y+1)$ and $k=2s$, where $x,y \geq 1$.\\
	We can write
	\begin{eqnarray}
	(K_u \times K_g)(2) &\cong& K_{u}(2) \times K_{s(2y+1)} \nonumber\\
	&\cong& (F_1 \oplus F_2 \oplus...\oplus F_u) \times  K_{s(2y+1)}, \ \mbox{by Theorem \ref{t1}} \nonumber\\
	&\cong& (F_1 \times  K_{s(2y+1)}) \oplus...\oplus (F_u \times  K_{s(2y+1)}), \nonumber
	\end{eqnarray}
	where each $F_l, \ 1 \leq l \leq u,$ consists of $2$ near $1$-factors of $K_u(2)$.
	Since each $F_l \times  K_{s(2y+1)} \cong 2\{x(K_2 \times  K_{s(2y+1)})\}$ and by Theorem \ref{op}, $K_{s(2y+1)}$ has a $C_s$-factorization
	 $\{{C}_s^1,{C}_s^2,...,{C}_s^{\frac{g-1}{2}}\},$
		where each ${C}_s^m, \ 1 \leq m \leq \frac{g-1}{2}$ is $C_s$-factor of $K_{s(2y+1)}$. We write $K_2 \times {C}_s^m \cong {C}_s^m \times K_2 \cong (2y+1)(C_s \times K_2)$ and by Theorem \ref{hc}, each $C_s \times K_2$ is a $C_{2s}$-factor.
	Thus, in total we get $(2x+1)(g-1)$ partial $C_{k}$-factors of $(K_u \times K_g)(2)$. Hence, $k$-ARCS of $(K_u \times {K}_g)(2)$exists.\\
	{\bf Case (ii): $1 < i \leq t-1$}\\
	Let $u=rx+1$, $g=s(2y+1)$ and $k=rs$, where $x,y \geq 1$.
	\begin{eqnarray}
	(K_u \times K_g)(2)& \cong& K_u(2) \times K_g \nonumber\\
	&\cong& (\mathcal{C}_{r}^1 \oplus \mathcal{C}_{r}^2 \oplus...\oplus \mathcal{C}_{r}^u)\times K_g, \ \ \mbox{by Theorem \ref{t4}} \nonumber\\
	&\cong&(\mathcal{C}_{r}^1 \times K_g )\oplus (\mathcal{C}_{r}^2\times K_g) \oplus...\oplus (\mathcal{C}_{r}^u \times K_g) \nonumber
	\end{eqnarray}
	where each $\mathcal{C}_{r}^j$, $1 \leq j \leq u$ is a near $C_{r}$-factor of $K_u(2)$.
	Since $ C_{r}^j \times K_g \cong x (C_{r} \times K_{s(2y+1)}),$
	by Theorem \ref{l1}, each $C_{r} \times K_{s(2y+1)}$ has $s(2y+1)-1$ $C_{rs}$ factors. The  $C_k$-factors obtained above become the partial $C_k$-factors of $(K_u \times K_g)(2)$. Thus, in total we get $(rx+1)(s(2y+1)-1)$ partial $C_k$-factors of $(K_u \times K_g)(2)$. Therefore $k$-ARCS of  $(K_u \times {K}_{g})(2)$ exists.
\end{proof}

\begin{thm} \label{pt2}
	Let $k=p_1p_2p_3...p_t \equiv 0(mod \ 4), r= p_1p_2p_3...p_i,$ and $s=p_{i+1}p_{i+2}p_{i+3}...p_t,\ 1 < i \leq t-1$ , where $p_1,p_2,...,p_t$ are primes not necessarily distinct. If $u=rx+1, \ x \in \mathbb{N} \setminus \{2\}$, $g \equiv s (mod \ 2 s)$ is odd,  then there exists a $k$-ARCS of $K_u \times {K}_{g}$.
\end{thm}
\begin{proof}
	 One can easily observe that $r=p_1p_2p_3...p_i \equiv 0(mod \ 4)$.\\
	Let $u=(4m)x+1,$  $g=s(2y+1)$ is odd, where $x \in \mathbb{N} \setminus \{2\},\ 4m=r, \ \ m \geq 1, \ y \geq 0$ and $k=rs$.  \\
	We establish the proof in two cases\\
	{\bf Case (i): $x=1$.}\\
	We can write 
	\begin{eqnarray}
	K_{r+1} \times K_g &\cong& K_{r+1} \times K_{s(2y+1)} \nonumber\\
	&\cong& \{(K_{r+1} \times K_{2y+1}) \otimes \bar{K}_s \}\oplus (2y+1)(K_{r+1} \times K_s) \label{4.8eq}
	\end{eqnarray}
	Graphs in the R.H.S of (\ref{4.8eq}) can be obtained by making $2y+1$ holes of type $K_{r+1} \times K_s$ in $K_{r+1} \times K_{s(2y+1)}$ and identifying the $s$-subsets of $K_{s(2y+1)}$ (in the resulting graph) into a single vertex and two of them are adjacent if the corresponding $s$-subsets form a $K_{s,s}$ in $K_{r+1} \times K_{s(2y+1)}.$ By expanding the vertices into $s$-subsets we get the first graph in  (\ref{4.8eq}).\\
	Now consider
	\begin{eqnarray}
	(K_{r+1} \times K_{2y+1}) \otimes \bar{K}_s &\cong& (\mathcal{C}_r^1\oplus...\oplus \mathcal{C}_r^{r+1}) \otimes \bar{K}_s, \ \mbox{by case (i) of Theorem \ref{tt3} } \nonumber\\
	&\cong& (\mathcal{C}_r^1 \otimes \bar{K}_s) \oplus...\oplus(\mathcal{C}_r^{rx+1} \otimes \bar{K}_s), \nonumber
	\end{eqnarray}
	where each $\mathcal{C}_r^l, \ 1 \leq l \leq r+1$, consists of $y$ partial $C_r$-factors of $K_{r+1} \times K_{2y+1}.$  Since each $\mathcal{C}_r^l \otimes \bar{K}_s
	\cong y\{(2y+1)(C_r \otimes \bar{K}_s)\}$ and by Theorem \ref{hc2}, each $C_r \otimes \bar{K}_s$ has $s$ $C_{rs}$-factors. Thus we get $ys$ partial $C_{rs}$-factors of $(K_{r+1} \times K_{2y+1}) \otimes \bar{K}_s$.\\
Let $A_0, A_1,...,A_{r-1},A_{\infty}$ be the parts of $K_{r+1} \times K_s.$
	By Theorem \ref{I2}, $K_{r+1} \times K_s$ has a partial $C_{rs}$-factorization $\mathcal{G}=\{G_{t,h} \cup G_{h}^{\infty}| t\in \mathbb{Z}_{r},   1 \leq h \leq \frac{s-1}{2}\}$ where 
	 each $G_{t,h}$ is a partial $C_{rs}$- factor missing the partite set $A_{t+\frac{r}{2}}$  and the subscripts of $A$ are taken modulo $r$;
	 each $G_{h}^{\infty}$ is a partial $C_{rs}$-factor missing the partite set $A_{\infty}.$\\
	Finally,  adding the partial $C_{rs}$-factors of (\ref{4.8eq}) gives the $\frac{s(2y+1)-1}{2}$
	partial $C_{k}$-factors of $K_{r+1} \times K_{s(2y+1)}.$
	Thus, in total we get $(r+1)(\frac{s(2y+1)-1}{2})$ partial $C_k$-factors of $K_{r+1} \times K_{s(2y+1)}.$
	Therefore $k$-ARCS of  $K_u \times {K}_{g}$ exists. \\
	{\bf Case (ii): $x >2.$}\\
	We can write 
	\begin{eqnarray}
	K_{rx+1} \times K_g &\cong& K_{rx+1} \times K_{s(2y+1)} \nonumber\\
	&\cong& \{(K_{rx+1} \times K_{2y+1}) \otimes \bar{K}_s \}\oplus (2y+1)(K_{rx+1} \times K_s) \label{4.81eq}
	\end{eqnarray}
	Graphs in the R.H.S of (\ref{4.81eq}) can be obtained by using a similar procedure as in (\ref{4.8eq}).\\
	Now consider
	\begin{eqnarray}
	(K_{rx+1} \times K_{2y+1}) \otimes \bar{K}_s &\cong& \{\mathcal{C}_r^1\oplus...\oplus \mathcal{C}_r^{rx+1}\} \otimes \bar{K}_s, \ \mbox{by case(ii) of  Theorem \ref{tt3}} \nonumber\\
	&\cong& (\mathcal{C}_r^1 \otimes \bar{K}_s) \oplus...\oplus(\mathcal{C}_r^{rx+1} \otimes \bar{K}_s), \nonumber
	\end{eqnarray}
	where each $\mathcal{C}_r^l, \ 1 \leq l \leq rx+1$, consists of $y$ partial $C_r$-factors of $K_{rx+1} \times K_{2y+1}.$  Since each $\mathcal{C}_r^l \otimes \bar{K}_s
	\cong y\{x(2y+1)(C_r \otimes \bar{K}_s)\}$ and by Theorem \ref{hc2}, each $C_r \otimes \bar{K}_s$ has $s$ $C_{rs}$-factors. Thus we get $ys$ partial $C_{rs}$-factors of $(K_{rx+1} \times K_{2y+1}) \otimes \bar{K}_s$.\\
	Now, we construct the partial $C_{rs}$-factors of $K_{rx+1} \times K_s$ as follows:\\
Using a similar procedure as in case (ii) of Theorem \ref{tt3}, we can write	
\begin{eqnarray}
K_{rx+1} \times K_s 
\cong \oplus_{i=0}^{x-1}\big\{(\Gamma'_i \times K_s) \oplus (\Gamma_i \times K_s)\big\},  \label{R3}
\end{eqnarray}
where $\Gamma_i$ and $\Gamma'_i$ are complete graph of order $4m+1$ and the complete multipartite graph $K_x \otimes \bar{K}_{4m}$ respectively.
\par By Theorem \ref{mn}, $\Gamma'_i$ has a $C_{4m}$-factorization $\mathcal{G}'_i=\{G'_{i,1},G'_{i,2},...,G'_{i,2m}\},$ and by Theorem \ref{l1}, $G'_{i,j} \times K_s$ has a $C_{rs}$-factorization $\mathcal{G}'_{i,j}=\{G'_{i,j,1},G'_{i,j,2},...,G'_{i,j,s-1}\},$ for every $i \in \mathbb{Z}_x,$ and $ j \in \{1,2,...,2m\}.$ Hence, $\cup_j \mathcal{G}'_{i,j}$ gives a $C_k$-factorization of $\Gamma'_i \times K_s.$
\par Considering that $\Gamma_i \times K_s \cong K_{r+1} \times K_s,$ let $A_0, A_1,...,A_{r-1},A_{\infty}$ be the parts of  $K_{r+1} \times K_s.$
By Theorem \ref{I2}, $\Gamma_i \times K_s (\cong K_{r+1} \times K_s)$ has a partial $C_k$-factorization $\mathcal{G}_i=\{G_{i,t,h} \cup G_{i,h}^{\infty}|i \in \mathbb{Z}_x,\  t\in \mathbb{Z}_{r},  \ 1 \leq h \leq \frac{s-1}{2}\}$ where 
each $G_{i,t,h}$ is a partial $C_k$-factor missing the partite set $A_{t+\frac{r}{2}}$  and the subscripts of $A$ are taken modulo $r$;
 each $G_{i,h}^{\infty}$ is a partial $C_k$-factor missing the partite set $A_{\infty}.$\\
One can check that $\mathcal{G}=\{G'_{i,j,l} | i \in \mathbb{Z}_x, 1 \leq j \leq 2m, 1 \leq l \leq s-1\} \cup \{G_{i,t,h} \cup G_{i,h}^{\infty}| i \in \mathbb{Z}_x, t\in \mathbb{Z}_{r},   1 \leq h \leq \frac{s-1}{2}\}$
 gives a partial $C_k$-factorization of $K_{rx+1} \times K_s.$ \\
Finally,  adding the partial $C_{k}$-factors of (\ref{4.81eq}) gives the $\frac{s(2y+1)-1}{2}$
partial $C_{k}$-factors of $K_{rx+1} \times K_{s(2y+1)}.$
Thus, in total we get $(rx+1)(\frac{s(2y+1)-1}{2})$ partial $C_k$-factors of $K_{rx+1} \times K_{s(2y+1)}.$
Therefore $k$-ARCS of  $K_u \times {K}_{g}$ exists.
	
\end{proof}

\begin{thm} \label{pt3}
	Let $k=p_1p_2p_3...p_t \geq 6$ be even, where $p_1,p_2,p_3,...,p_t$ are primes not necessarily distinct. If $u \equiv 1 (mod \ p_1p_2p_3...p_i )$ and $g \equiv 0 (mod \  p_{i+1}p_{i+2}p_{i+3}...p_t)$, $1 \leq i \leq t-1,$ $p_1,p_2,p_3,...,p_i$ are odd primes and at least one of  the $p_{i+1},p_{i+2},p_{i+3},...,p_t$ is $2$, then there exists a $k$-ARCS of $(K_u \times {K}_{g})(2)$.
\end{thm}
	\begin{proof}
		Suppose that $r=p_1p_2...p_i$ and $s=p_{i+1}p_{i+1}...p_t$, $1 \leq i \leq t-1$.
		Let $u=rx+1$, $g=sy$, where $x,y \geq 1$ and $k=rs$, $1 \leq i \leq t-1$.\\
		Now, we write
		\begin{eqnarray}
		(K_u \times K_g)(2)&\cong&K_{rx+1}(2) \times K_g \nonumber \\
		&\cong&(\mathcal{C}_{r}^1 \oplus ... \oplus \mathcal{C}_{r}^u) \times K_g, \ \mbox{by Theorem \ref{t5}} \nonumber \\
		&\cong&(\mathcal{C}_{r}^1 \times K_{sy} ) \oplus ... \oplus (\mathcal{C}_{r}^u \times K_{sy} ), \nonumber
		\end{eqnarray}
		where each $\mathcal{C}_{r}^j$, $1 \leq j \leq u,$ is a near ${C}_{r}$-factor of $K_u(2)$.
		Since $\mathcal{C}_{r}^j \times K_{sy}
		\cong x (C_{r} \times K_{sy}),$
		by Theorem \ref{l1}, for each $C_{r} \times K_{sy}$ has $sy-1$ $C_{rs}$-factors. The $C_{rs}$-factors obtained above become the partial $C_{rs}$-factors of $ (K_{rx+1} \times K_{sy})(2)$. Thus, in total we get $(rx+1)(sy-1)$ partial $C_{rs}$-factors of $ (K_{rx+1} \times K_{sy})(2).$ Therefore $k$-ARCS of  $(K_u \times {K}_{g})(2)$ exists.
	\end{proof}
\begin{thm} \label{both even}
	Let $k=p_1p_2p_3...p_t \equiv 0(mod \ 4)$, where $p_1,p_2,p_3,..,p_t$ are primes not necessarily distinct. If $u \equiv 1 (mod \ p_1p_2p_3...p_i )$ and $g \equiv 0 (mod \  p_{i+1}p_{i+2}p_{i+3}...p_t)$, $1 \leq i < t-1,$ at least one of the $p_1,p_2,...,p_i$ and $p_{i+1},p_{i+2},...,p_t$ is $2$, then there exists a $k$-ARCS of $(K_u \times {K}_{g})(2)$.
\end{thm}
\begin{proof}
Suppose that $r=p_1p_2...p_i$ and $s=p_{i+1}p_{i+1}...p_t$, $1 \leq i \leq t-1$.\\
We establish the proof in two cases\\
{\bf Case (i)} $i=1$ and $p_1=2$.\\
Let $u=2x+1$, $g=2sy$, where $x,y \geq 1$ and $k=2s.$\\
This is substantially the same proof as in case (i) of Theorem \ref{pt1}. The only difference is that we take $s$ is even.\\
{\bf Case (ii)} $1 <i < t-1$ \\
Let $u=rx+1$, $g=sy$, where $x,y \geq 1$ and $k=rs$.
 We can write
\begin{eqnarray}
(K_u \times K_g)(2)&\cong&K_{rx+1}(2) \times K_g \nonumber \\
&\cong&(\mathcal{C}_{r}^1 \oplus ... \oplus \mathcal{C}_{r}^u) \times K_g, \ \mbox{by Theorem \ref{t4}} \nonumber \\
&\cong&(\mathcal{C}_{r}^1 \times K_{sy} ) \oplus ... \oplus (\mathcal{C}_{r}^u \times K_{sy} ), \nonumber
\end{eqnarray}
where each $\mathcal{C}_{r}^j$, $1 \leq j \leq u,$ is a near ${C}_{r}$-factor of $K_u(2)$.
Since $\mathcal{C}_{r}^j \times K_{sy}
\cong x (C_{r} \times K_{sy}),$ and we know that $C_r \times K_{sy} \cong (C_r \times K_y) \otimes \bar{K}_s \oplus y(C_r \times K_s).$
By Theorem \ref{l1},  $C_{r} \times K_{y}$ has a $y-1$ $C_{r}$-factors, and by Theorem \ref{hc2}, each $C_r \otimes \bar{K}_s$ has a $s  C_{rs}$-factors. By Theorem \ref{pten}, $C_r \times K_s$ has a $C_{rs}$-factors. Thus we have obtained $(sy-1) \ C_{k}$-factors of $C_r \times K_{sy}$.
The $C_{k}$-factors obtained above become the partial $C_{k}$-factors of $ (K_{rx+1} \times K_{sy})(2)$. Thus, in total we get $(rx+1)(sy-1)$ partial $C_{k}$-factors of $ (K_{rx+1} \times K_{sy})(2).$ Therefore $k$-ARCS of  $(K_u \times {K}_{g})(2)$ exists.
\end{proof}

\begin{proof}[\bf Proof of Theorem \ref{mt}]
Let $k=p_1p_2p_3...p_t$, where $p_1,p_2,p_3,...,p_t$ are primes not necessarily distinct.\\
Necessity follows from Theorem \ref{nsty}. Sufficiency can be divided into two cases.\\
 {\bf Case (i):} When $\lambda=1$,
  the values of $u$ and $g$ are one of the following
\begin{itemize}
\item[(a)] $u=kx+1$, $g \geq 3$ is odd, where $x \in \mathbb{N} \setminus \{2\};$
\item[(b)] $ u=rx+1, \ g \equiv s(mod  \ 2s)$ is odd, where $r=p_1p_2p_3...p_i \equiv 0(mod \ 4), \ s=p_{i+1}p_{i+2}...p_t, \ 1 <i \leq t-1, \ x \in \mathbb{N} \setminus \{2\};$ 
\end{itemize}
{\bf Case (ii):} When $\lambda=2$,
the values of $u$ and $g$ are one of the following
\begin{itemize}
\item[(c)] $u \equiv 1(mod \ k)$, $g \geq 2;$ 
\item[(d)] $u \geq 3$ is odd, $g \equiv 0(mod \ k);$ 
\item[(e)] $u=4x$, $g=ky$, $x \in \mathbb{N} \setminus \{2\}, \ y \geq 1$ is odd,  where $k \geq 6;$ 
\item[(f)] $u=4x$, $g=ky$, $x \in \mathbb{N} \setminus \{2\}, \ y  \geq 2$ is even,  where $k > 6$; 
\item[(g)] $u \equiv 1 (mod \ p_1p_2p_3...p_i)$, $g \equiv 0(mod  \ 2p_{i+1}p_{i+2}...p_t)$, $1 \leq i < t-1$, where at least one of the  $p_1,p_2,...,p_i$ is $2$ and $p_{i+1},p_{i+2},...,p_t$ are odd primes;
\item[(h)] $u  \equiv 1 (mod \ p_1p_2p_3...p_i)$, $g \equiv 0 (mod  \ p_{i+1}p_{i+2}...p_t)$, $1 \leq i \leq t-1$, where $p_1,p_2,...,p_i$ are odd primes and at least one of the $p_{i+1},p_{i+2},...,p_t$  is $2;$ 
\item[(i)] $u  \equiv 1 (mod \ p_1p_2p_3...p_i)$, $g \equiv 0 (mod  \ p_{i+1}p_{i+2}...p_t)$, $1 \leq i < t-1$, where at least one of $p_1,p_2,...,p_i$ and $p_{i+1},p_{i+2},...,p_t$ is $2$.
\end{itemize}	
The proof for  (a),(b),(c),(d),(e),(f),(g),(h), and (i) follows from Theorems  \ref{tt3}, \ref{pt2},\ref{tt2}, \ref{tt1}, \ref{tt4}, \ref{26}, \ref{pt1}, \ref{pt3}, and  \ref{both even}.\\
 		If $\lambda >2$ is odd (respectively, even), the values for $u$ and $g$ are the same as that of case(i) (respectively, case(i) and case (2)).  Hence  $k$-ARCS of $(K_u \times {K}_g)(\lambda)$ exists.
	\end{proof}

\begin{rmk*}
	For both even $\lambda \geq 2, \ k \equiv 2(mod \ 4),$	there exists a $k$-ARCS of $(K_u \times {K}_{g})(\lambda)$ if and only if $u \geq 3$, $g\geq 2$, $\lambda(g-1) \equiv 0 (mod \ 2)$, $g(u-1) \equiv 0 (mod \ k)$  except possibly $(u,k) \in \{(8,4s+2),(4t+2,4s+2)| s,t \geq 1\}$ and $(u,g,k) \in \{(4t,6y,6)| \ t \geq 1 \ \mbox{and} \ y \geq 2 \ \mbox{is even}\}.$
\end{rmk*}
When $\lambda=2$ proof for the above statement is immediately follows from Theorems \ref{tt2}, \ref{tt1} - \ref{pt1}, \ref{pt13}, and \ref{pt3}. If $\lambda >2$ is even, the values for $u$ and $g$ are the same as that of Theorems \ref{tt2}, \ref{tt1} - \ref{pt1}, \ref{pt13}, and \ref{pt3}.  Hence  $k$-ARCS of $(K_u \times {K}_g)(\lambda)$ exists.
\section{Conclusion}
\par  In this paper, we have settled the existence of  $k$-ARCS of $(K_u \times {K}_{g})(\lambda)$, for $k \equiv 0 (mod \ 4)$ with  few possible exceptions. Our results also gives a partial solution to the existence of modified cycle frames of complete multipartite multigraphs. 
\section*{Acknowledgments:}
\noindent First author thank the Periyar University for its support through URF grant no: PU/ AD-3/URF/015723/2020.\\
Second author thank NBHM for its support through PDF grant no: 2/40(22)/2016/R\&D-11/15245.\\
Corresponding author thank  DST, New Delhi for its support through FIST grant no: SR/FIST/MSI-115/2016 (Level-I).

\end{document}